\documentclass{amsart}

\usepackage{amsmath,amssymb,amsfonts,enumerate,amsthm, amscd}

\renewcommand{\Im}{\mbox{Im}\,}

\newcommand{\pd}{\mbox{pd}\,}
\newcommand{\id}{\mbox{id}\,}
\newcommand{\fd}{\mbox{fd}\,}
\newcommand{\gd}{\mbox{Gdim}\,}
\newcommand{\wdim}{\mbox{wdim}\,}
\newcommand{\gldim}{\mbox{gldim}\,}

\newtheorem{theorem}{Theorem}[section]
\newtheorem{lemma}[theorem]{Lemma}
\newtheorem{proposition}[theorem]{Proposition}
\newtheorem{corollary}[theorem]{Corollary}

\theoremstyle{definition}
\newtheorem{definition}[theorem]{Definition}
\newtheorem{definitions}[theorem]{Definitions}
\theoremstyle{remark}
\newtheorem{remark}[theorem]{Remark}
\newtheorem{example}[theorem]{Example}

\theoremstyle{Definition and Notation}

\begin{document}
\bibliographystyle{amsplain}

%\date{}
\title[Strongly $n$-Gorenstein...]{Strongly $n$-Gorenstein projective, injective
and flat modules}

\author{Najib Mahdou}
\address{Najib Mahdou\\Department of Mathematics, Faculty of Science and Technology of Fez, Box 2202, University S.M. Ben Abdellah Fez, Morocco.}

\author{Mohammed Tamekkante}
\address{Mohammed Tamekkante\\Department of Mathematics, Faculty of Science and Technology of Fez, Box 2202, University S.M. Ben Abdellah Fez, Morocco.}



\begin{abstract}

This paper generalize the idea of the authors in \cite{Bennis and
Mahdou1}. Namely, we define and  study a particular case of modules
with Gorenstein projective, injective, and flat dimension less or
equal than $n\geq 0$ , which we call, respectively, strongly
n-Gorenstein projective, injective and flat modules. These three
classes of modules give us a new characterization of the first
modules, and they are a generalization of the notions of strongly
Gorenstein projective, injective, and flat modules respectively.

\end{abstract}

\maketitle
%%%%%%%%%%%%%%%%%%%%%%%%%%%%%%%%%%%%%%%%%%%%%%%%%%%%%%%%%
%%%%%%%%%%%%%%%%%%%%%%%%%%%%%%%%%%%%%%%%%%%%%%%%%%%%%%%%%
%%Introduction%%%
%%%%%%%%%%%%%%%%%%%%%%%%%%%%%%%%%%%%%%%%

\section{Introduction}

Throughout this paper, all rings are commutative with identity
element, and all modules are unital.

Let $R$ be a ring, and let $M$ be an $R$-module. As usual we use
$\pd_R(M)$, $\id_R(M)$ and $\fd_R(M)$ to denote, respectively, the
classical projective dimension, injective dimension and flat
dimension of $M$. By $\gldim(R)$ and $\wdim(R)$ we denote,
respectively, the classical global dimension and weak dimension of
R. It is convenient to use ``local" to refer to (not necessarily
Noetherian) rings with a unique maximal ideal.

For a two-sided Noetherian ring $R$, Auslander and Bridger \cite{Aus
bri} introduced the $G$-dimension, $\gd_R (M)$, for every finitely
generated $R$-module $M$. They proved the inequality $\gd_R (M)\leq
\pd_R (M)$ with equality $\gd_R (M) = \pd_R (M)$ when $\pd_R (M)$ is
finite.

Several decades later, Enochs and Jenda \cite{Enochs,Enochs2}
defined the notion of Gorenstein projective dimension
($G$-projective dimension for short), as an extension of
$G$-dimension to modules that are not necessarily finitely
generated, and the Gorenstein injective dimension ($G$-injective
dimension for short) as a dual notion of Gorenstein projective
dimension. Then, to complete the analogy with the classical
homological dimension, Enochs, Jenda and Torrecillas \cite{Eno Jenda
Torrecillas} introduced the Gorenstein flat dimension. Some
references are
 \cite{Christensen, Christensen
and Frankild, Enochs, Enochs2, Eno Jenda Torrecillas}.

In 2004, Holm \cite{Holm} generalized several results which are
already obtained over Noetherian rings to associative rings.

Recently in \cite{Bennis and Mahdou2}, the authors started the study
of global Gorenstein dimensions of rings, which are called, for a
commutative ring $R$, projective, injective, and weak dimensions of
$R$, denoted by $GPD(R)$, $GID(R)$, and $G.wdim(R)$, respectively,
and, respectively, defined as follows:\bigskip

$\begin{array}{cccc}
  1) & GPD(R) & = & sup\{ Gpd_R(M)\mid M$ $R-module\} \\
  2) & GID(R) & = & sup\{ Gid_R(M)\mid M$ $R-module\} \\
  3) & G.wdim(R) & = & sup\{ Gfd_R(M)\mid M$ $R-module\}
\end{array}$
\\

They proved that, for any ring R, $ G.wdim(R)\leq GID(R) = GPD(R)$
(\cite[Theorems 2.1 and 2.11]{Bennis and Mahdou2}). So, according to
the terminology of the classical theory of homological dimensions of
rings, the common value of $GPD(R)$ and $GID(R)$ is called
Gorenstein global dimension of $R$, and denoted by $G.gldim(R)$.\\
They also proved that the Gorenstein global and weak dimensions are
refinement of the classical global  and weak dimensions of rings.
That is : $G.gldim(R) \leq gldim(R)$ and $G.wdim(R)\leq wdim(R)$
with equality if $wdim(R)$ is finite (\cite[Propositions
2.12]{Bennis and Mahdou2}).\bigskip

In \cite{Bennis and Mahdou1}, the authors introduced a particular
case of Gorenstein projective, injective, and flat modules, which
are defined, respectively, as follows:
\begin{definitions}\
\begin{enumerate}
    \item A module $M$ is said to be strongly Gorenstein projective ($SG$-projective for short), if
there exists an exact sequence of   projective modules of the form:
$$\mathbf{P}=\ \cdots\rightarrow P\stackrel{f}\rightarrow
P\stackrel{f} \rightarrow P\stackrel{f}
         \rightarrow P \rightarrow\cdots$$ such that  $M \cong \Im(f)$ and such that $Hom(-,Q)$ leaves $\mathbf{P}$ exact
whenever $Q$ is a projective module.

The exact sequence $\mathbf{P}$ is called a strongly complete
projective resolution and denoted by $(P,f)$.
    \item The strongly Gorenstein injective module
is defined dually.
    \item A module $M$ is said to be strongly
Gorenstein flat ($SG$-flat for short), if there exists an exact
sequence of flat modules  of the form:
$$\mathbf{F}=\ \cdots\rightarrow F\stackrel{f}\rightarrow F \stackrel{f}\rightarrow
F \stackrel{f}\rightarrow F \rightarrow\cdots$$ such that  $M \cong
\Im(f)$ and such that $I\otimes-$ leaves $\mathbf{F}$ exact whenever
$I$ is an injective module. The exact sequence $\mathbf{F}$ is
called a strongly complete flat resolution and denoted by $(F,f)$.
\end{enumerate}
\end{definitions}
The principal role of the strongly Gorenstein projective and
injective modules is to give a simple characterization of Gorenstein
projective and injective modules, respectively, as follows:
\begin{theorem}[\cite{Bennis and Mahdou1},Theorem 2.7] A  module is Gorenstein projective (resp., injective)
if, and only if, it is a direct summand of a strongly Gorenstein
projective (resp., injective) module.

\end{theorem}
Using \cite[Theorem 3.5]{Bennis and Mahdou1} together with
\cite[Theorem 3.7]{Holm}, we have the next result:
\begin{proposition} Let $R$ be a coherent ring. A module is Gorenstein flat
if, and only if, it is a direct summand of a strongly Gorenstein
flat module.
\end{proposition}

 This result allows us to show that the strongly Gorenstein
projective, injective and flat modules have simpler
characterizations than their Gorenstein correspondent modules. For
instance:
\begin{theorem}[\cite{Bennis and Mahdou1}, Propositions 2.9 and
3.6]\label{caracterization of SG}\
\begin{enumerate}
    \item  A module M is strongly Gorenstein
projective if, and only if, there exists a short exact sequence of
modules: $0\longrightarrow M \longrightarrow P \longrightarrow M
\longrightarrow 0$ where $P$ is projective and $Ext(M,Q) = 0$ for
any projective module $Q$.
    \item A module M is strongly Gorenstein
injective if, and only if, there exists a short exact sequence of
modules: $0\longrightarrow M \longrightarrow I \longrightarrow M
\longrightarrow 0$ where $I$ is injective and $Ext(E,M) = 0$ for any
injective module $E$.
    \item A module M is strongly Gorenstein
flat  if, and only if, there exists a short exact sequence of
modules: $0\longrightarrow M \longrightarrow F \longrightarrow M
\longrightarrow 0$ where $F$ is flat and $Tor(M,I) = 0$ for any
injective  module $I$.
\end{enumerate}
\end{theorem}

Along this paper we need the following Lemmas:

\begin{lemma}\label{lemma1}
Let $0\rightarrow N\rightarrow N' \rightarrow N'' \rightarrow 0$ be
an exact sequence of $R$-modules. Then:
\begin{enumerate}
  \item $Gpd_R(N)\leq max\{Gpd_R(N'),Gpd_R(N'')-1\}$ with equality
  if $Gpd_R(N')\neq Gpd_R(N'')$.
  \item $Gpd_R(N')\leq max\{Gpd_R(N),Gpd_R(N")\}$ with equality
  if $Gpd_R(N'')\neq Gpd_R(N)+1$.
  \item $Gpd_R(N'')\leq max\{Gpd_R(N'),Gpd_R(N)+1\}$ with equality
  if $Gpd_R(N')\neq Gpd_R(N)$.
\end{enumerate}
\end{lemma}
\begin{proof}
Using \cite[Theorems 2.20 and 2.24]{Holm} the argument is analogous
to the one of \cite[Corollary 2, p. 135]{bourbaki}.
\end{proof}
Dually we have:
\begin{lemma}\label{injective version}
Let $0\rightarrow N\rightarrow N' \rightarrow N'' \rightarrow 0$ be
an exact sequence of $R$-modules. Then:
\begin{enumerate}
  \item $Gid_R(N)\leq max\{Gid_R(N'),Gid_R(N'')+1\}$ with equality
  if $Gid_R(N')\neq Gid_R(N'')$.
  \item $Gid_R(N')\leq max\{Gid_R(N),Gid_R(N'')\}$ with equality
  if $Gid_R(N'')+1\neq Gid_R(N)$.
  \item $Gid_R(N'')\leq max\{Gid_R(N'),Gid_R(N)-1\}$ with equality
  if $Gid_R(N')\neq Gid_R(N)$.
\end{enumerate}
\end{lemma}
And using \cite[Proposition 3.11]{Holm} and Lemma \ref{injective
version} we get the following Lemma

\begin{lemma}\label{lemma3}
Let $0\rightarrow N\rightarrow N' \rightarrow N'' \rightarrow 0$ be
an exact sequence of modules over a coherent ring $R$. Then:
\begin{enumerate}
  \item $Gfd_R(N)\leq max\{Gfd_R(N'),Gfd_R(N'')-1\}$ with equality
  if $Gfd_R(N')\neq Gfd_R(N'')$.
  \item $Gfd_R(N')\leq max\{Gfd_R(N),Gfd_R(N'')\}$ with equality
  if $Gfd_R(N'')\neq Gfd_R(N)+1$.
  \item $Gfd_R(N'')\leq max\{Gfd_R(N'),Gfd_R(N)+1\}$ with equality
  if $Gfd_R(N')\neq Gfd_R(N)$.
\end{enumerate}
\end{lemma}

In \cite{Holm}, Holm gives a characterization of modules with finite
Gorenstein projective, injective and flat modules (\cite[Theorems
2.20, 2.22 and 3.14]{Holm}). In this three characterizations, Holm
impose the finitely of this dimensions. Almost by definition one has
the inclusion
$$\{M| pd(M)\leq n\}\subseteq \{M|Gpd(M)\leq n\}.$$
The main idea of this paper is to introduce and study an
intermediate class of modules called strongly $n$-Gorenstein
projective modules. Similarly, we define the strongly $n$-Gorenstein
injective and flat modules.\\
The simplicity of these modules manifests in the fact that they has
simpler characterizations than their corresponding Gorenstein
modules. Moreover, with such modules, we are able to give nice new
characterizations of modules with Gorenstein projective, injective
and flat dimensions equal to $n$.

%%%%%%%%%%%%%%%%%%%%%%%%%%%%%%%%%%%%%%%%%%%%%%%%%%%%%%%%%
%%%%%%%%%%%%%%%%%%%%%%%%%%%%%%%%%%%%%%%%%%%%%%%%%%%%%%%%%
%%Section1%%%%%%%%%%%%%%%%%%%%%%%%%%%%%%%%%%%%%%%%%%%
%%%%%%%%%%%%%%%%%%%%%%%%%%%%%%%%%%%%%%%%

\begin{section}{Strongly $n$-Gorenstein projective and injective modules}

In this section, we introduce and study strongly $n$-Gorenstein
projective and injective modules which are defined as follows:
\begin{definitions}
Let $n$ be a positive integer.
\begin{enumerate}
  \item An $R$-module $M$ is said to be strongly $n$-Gorenstein projective,
if there exists a short exact sequence
$$0\longrightarrow M \longrightarrow P \longrightarrow M
\longrightarrow 0$$ where $pd(P)\leq n$ and $Ext^{n+1}(M,Q)=0$
whenever $Q$ is projective.
  \item An $R$-module $M$ is said to be strongly $n$-Gorenstein injective,
if there exists a short exact sequence
$$0\longrightarrow M \longrightarrow I \longrightarrow M
\longrightarrow 0$$ where $id(I)\leq n$ and $Ext^{n+1}(E,M)=0$
whenever $E$ is injective.
\end{enumerate}
\end{definitions}

A direct consequence of the above definition is that, the strongly
$0$-Gorenstein projective modules are just the strongly Gorenstein
projective modules (by \cite[Proposition 2.9]{Bennis and Mahdou1}).
Also every module with finite projective dimension less or equal
than $n$ is a strongly $n$-Gorenstein projective module and  we
have:
\begin{proposition}\label{n-stron implis n-Gorenstein}
Let $n$ be a positive integer and  $M$ be a strongly $n$-Gorenstein
projective module. Then, the following hold:
\begin{enumerate}
  \item If $0\rightarrow N \rightarrow P_n\rightarrow ...
\rightarrow P_1 \rightarrow M \rightarrow 0$ is an exact sequence
where all $P_i$ are projective, then $N$ is strongly Gorenstein
projective module and consequently $Gpd(M)\leq n$.
  \item Moreover, if $0\longrightarrow M \longrightarrow P \longrightarrow M
\longrightarrow 0$ is a short exact  sequence where $pd(P)< \infty$
then $Gpd(M)=pd(P)$ and consequently $M$ is strongly $k$-Gorenstein
projective module with $k:=pd(P)$.
\end{enumerate}

\end{proposition}
\begin{proof}
$(1)$ If $n=0$ the result holds from \cite[Proposition 2.9]{Bennis
and Mahdou1}. Otherwise, since $M$ is strongly $n$-Gorenstein
projective module, there is a short exact sequence $$0\rightarrow M
\rightarrow P \rightarrow M \rightarrow 0$$ where $pd(P)\leq n$.
Consider the following $n$-step projective resolution of $M$:
$$0\rightarrow N \rightarrow P_n\rightarrow ...
\rightarrow P_1 \rightarrow M \rightarrow 0$$ Hence, there is a
module $Q$ such that  the following diagram is commutative:
$$\begin{array}{ccccccccc}
   & 0 &  & 0 &  & 0 &  & 0 &  \\
  & \downarrow &  & \downarrow &  & \downarrow &  & \downarrow &  \\
  0\rightarrow  & N & \rightarrow & P_n & \rightarrow ...\rightarrow & P_1 &\rightarrow& M& \rightarrow 0 \\
   & \downarrow &  & \downarrow &  & \downarrow &  & \downarrow &  \\
    0\rightarrow  & Q & \rightarrow & P_n\oplus P_n & \rightarrow ...\rightarrow & P_1\oplus P_1 &\rightarrow& P& \rightarrow 0 \\
  & \downarrow &  & \downarrow &  & \downarrow &  & \downarrow &  \\
  0\rightarrow  & N & \rightarrow & P_n & \rightarrow ...\rightarrow & P_1 &\rightarrow& M& \rightarrow 0 \\
  & \downarrow &  & \downarrow &  & \downarrow &  & \downarrow &  \\
 & 0 &  & 0 &  & 0 &  & 0 &  \\
\end{array}$$
Clearly, $Q$ is projective since $pd(P)\leq n$ and  for every
projective module $K$, $Ext(N,K)=Ext^{n+1}(M,K)=0$.  Thus, by
\cite[Proposition 2.9]{Bennis and Mahdou1}, $N$ is strongly
Gorenstein projective module (then, Gorenstein projective). So,
$Gpd(M)\leq n$.\\

$(2)$ From the short exact sequence $0\longrightarrow M \rightarrow
P \rightarrow M \rightarrow 0$ and \cite[Proposition 2.27]{Holm} and
Lemma \ref{lemma1} and since $Gpd(M)$ is finite by $(1)$ above, we
have
$$k:=pd(P)=Gpd(P)= max\{Gpd(M),Gpd(M)\}=Gpd(M)$$
Thus, $Gpd(M)=pd(P)$. By \cite[Theorem 2.20]{Holm},
$Ext^{k+1}(M,K)=0$ whenever $K$ is projective. Consequently, $M$ is
strongly $k$-Gorenstein projective module.
 \end{proof}
Using \cite[Theorem 2.20]{Holm}, a direct consequence of Proposition
\ref{n-stron flat implis n-Gorenstein}, is that every strongly
$n$-Gorenstein projective module is strongly $m$-Gorenstein module
whenever $n\leq m$.\\

\begin{proposition}\label{sum and product}\
\begin{enumerate}
  \item If $(M_i)_{i\in I}$ is a family of strongly $n$-Gorenstein projective modules, then $\bigoplus M_i$ is strongly $n$-Gorenstein
projective.
  \item If $(M_i)_{i\in I}$ is a family of strongly $n$-Gorenstein injective modules, then $\prod M_i$ is strongly $n$-Gorenstein
injective.
\end{enumerate}
\end{proposition}
\begin{proof}
Clear since $pd(\oplus  M_i)=sup\{pd(M_i)\}$ and $id(\prod
M_i)=sup\{id(M_i)\}$  and also since $Ext^i(\bigoplus M_i, N)\cong
\bigoplus Ext^i(M_i,N)$ and $Ext^i(M,\prod N_i)\cong \prod
Ext^i(M,N_i)$ for every modules $M,N,M_i,N_i$ and all $i\geq 0$.
\end{proof}

It is clear that, for a positive integer $n$ and an $R$-module $M$:
$$``pd(M)\leq n"\Longrightarrow ``M\; is\; strongly\; n-Gorenstein\;
projective" \Longrightarrow ``Gpd(M)\leq n"$$

The converse is false as the following two examples shows:

\begin{example}\label{M sgp dont implies proje}
Consider the quasi-Frobenius local ring $R:=K[X]/(X^2)$ where $K$
is a field and we  denote by $\overline{X}$  the residue class in
$R$ of $X$. Let $S$ be a Noetherian ring such $gldim(S)=n$.
Consider a finitely generated $S$-module $M$ of $S$ such that
$pd_{S}(M) =n$. Set $T =R \times S$ and set $E :=(\overline{X})
\times M$. Then:
\begin{enumerate}
  \item $E$ is strongly $n$-Gorenstein projective $T$-module and
$Gpd_{T}(E) =n$.
  \item However,
  $pd_{T}(E) =\infty .$
\end{enumerate}
\end{example}
\begin{proof}
$(1)$  Consider the short exact sequence of $R$-modules:
$$0\rightarrow (\overline{X})\stackrel{\kappa}\rightarrow
R\stackrel{\phi}\rightarrow (\overline{X})\rightarrow 0$$ where
$\kappa$ is the injection and  $\phi$ is the multiplication by
$\overline{X}$. And consider also the short exact sequence of
$S$-module:
$$0\rightarrow M \stackrel{\iota}\rightarrow M\oplus M\stackrel{\pi}\rightarrow M \rightarrow 0$$
where $\iota$ and $\pi$ are respectively the canonical injection and
projection. Hence, we have the short exact sequence of $R\times
S$-module: $$(\star)\qquad 0\rightarrow E\rightarrow R\times
(M\oplus M)\rightarrow E\rightarrow 0$$ By \cite[Lemma 2.5(2)]{Costa
Mahdou}, $pd_T(R\times (M\oplus M))=pd_S(M\oplus M)=n$. On the other
hand, by \cite[Theorem 3.1]{Bennis and Mahdou3} and
\cite[Propositions 2.8 and 2.12]{Bennis and Mahdou2}, we have
$$G.gldim(T)=sup\{G.gldim(R),G.gldim(S)\}=gldim(S)=n<\infty$$ Then,
$Gpd_T(E)<\infty$. Therefore, applying  Lemma \ref{lemma1} to
$(\star)$, $$Gpd_T(E)\leq max\{Gpd_T(R\times (M\oplus
M)),Gpd_T(E)-1\}$$ Thus, $Gpd_T(E)\leq Gpd_T(R\times (M\oplus M))$.
Using Lemma \ref{lemma1} again to $(\star)$, we have $$Gpd_T(R\times
(M\oplus M))\leq max\{Gpd_T(E),Gpd_T(E)\}=Gpd_T(E)$$ So,
$Gpd_T(E)=Gpd_T(R\times (M\oplus M))$. On the other hand, by
\cite[Propostion 2.27]{Holm}, $Gpd_T(R\times (M\oplus
M))=pd_T(R\times (M\oplus M))=n$. Consequently, $Gpd_T(E)=n$ and by
$(\star)$ and \cite[Theorem 2.20]{Holm}, $E$ is strongly
$n$-Gorenstein projective $T$-module, as desired.\\

$(2)$ Using \cite[Lemma 2.5(2)]{Costa Mahdou},
$pd_T(E)=sup\{pd_R(\overline{X}),pd_S(M)\}$. Now, suppose that
$pd_R(\overline{X})<\infty$. Thus, by \cite[Proposition 2.8 and
Corollary 2.10]{Bennis and Mahdou2}, $\overline{X}$ is projective
 and then free since $R$ is
  local. Absurd, since $\overline{X}^2=0$. Consequently,
  $pd_T(E)=\infty$.

\end{proof}

\begin{example}\label{Gp don't implis SGp}
Consider the Noetherian local ring $R:=K[[X,Y]]/(XY)$ where $K$ is
a field, and we denote by $\overline{X}$  the residue class in $R$
of $X$. Let $S$ be a Noetherian ring such that $gldim(S)=n$. Let
$M$ be a finitely generated  $S$-module such that $pd_{S}(M) =n$.
Set $T =R \times S$ and set $E :=(\overline{X}) \times M$. Then:
\begin{enumerate}
  \item $Gpd_{T}(E)=n$.
  \item There is no positive
  integer $k$ for such  $E$ is  strongly $k$-Gorenstein
  $T$-module.
\end{enumerate}
\end{example}
\begin{proof}
$(1)$ By \cite[Lemma 3.2]{Bennis and Mahdou3} and \cite[Theorem
2.27]{Holm},
  $$n=pd_S(M)=Gpd_S(M)=Gp_S(E\otimes_T S)\leq Gpd_T(E)$$
  On the other hand, seen \cite[Propostions 2.8 and 2.10 and 2.12 ]{Bennis and
  Mahdou2}, the conditions of \cite[Lemma 3.3]{Bennis and Mahdou3}
  are satisfied. Hence, we have
  $$Gpd_T(E)\leq sup\{Gpd_R(\overline{X}),Gpd_S(M)\}=pd_S(M)=n$$
  Consequently, $Gpd_T(E)=n$, as desired.\\

$(2)$ Suppose the existence of a positive integer $k$ such that $E$
is strongly $k$-Gorenstein projective $T$-module. Then, there exist
a short exact sequence of $T$-modules $0\rightarrow E \rightarrow P
\rightarrow E \rightarrow 0$ where $pd_T(P)<\infty$. Since $R$ is a
projective $T$-module and since $(\overline{X})\cong_RE\otimes_TR$
we have a short exact sequence of $R$-modules
$$0\rightarrow (\overline{X})\rightarrow P\otimes_TR\rightarrow (\overline{X})\rightarrow
0$$ Notice that $pd_R(P\otimes_TR)<\infty$ since $R$ is a projective
$T$-module. Using \cite[Propositions 2.8 and 2.10]{Bennis and
Mahdou2}, we get that $P\otimes_TR$ is a projective $R$-module and
that $(\overline{X})$ is a Gorenstein projective $R$-module. So, by
\cite[Theorem 2.20]{Holm}, $(\overline{X})$ is strongly Gorenstein
projective module. Absurd  (by \cite[Example 2.13(2)]{Bennis and
Mahdou1}).
\end{proof}

Now we give our main result of this paper.

\begin{theorem}\label{direct summand SGP}
Let $M$ be an $R$-module and $n$ a positive  integer. Then,
$Gpd_R(M)\leq n$ if, and only if, $M$ is a direct summand of a
strongly $n$-Gorenstein projective module.
\end{theorem}
\begin{proof}
 If $n=0$ the result holds from \cite[Theorem 2.7]{Bennis and
Mahdou1}. So, assume that  $0<Gpd(M)\leq n$. From, \cite[Theorem
2.10]{Holm}, there is an exact sequence of $R$-module $0\rightarrow
K \rightarrow G \rightarrow M$  where $G$ is Gorenstein  projective
and $pd(K)\leq n-1$. By definition of Gorenstein projective module
there is a short exact sequence
$$0\rightarrow G \rightarrow P \rightarrow G^0\rightarrow 0$$
 where $P$ is projective
and $G'$ is Gorenstein projective. Hence, consider the following
pushout diagram:
$$\begin{array}{ccccccc}
   & 0 &  & 0 &  &  &  \\
   & \downarrow &  & \downarrow &  & &  \\
   & K& = & K &  &  &  \\
   & \downarrow &  & \downarrow &  &  &  \\
  0\rightarrow & G & \rightarrow & P  & \rightarrow & G^0 & \rightarrow 0 \\
   & \downarrow &  & \downarrow &  & \parallel &  \\
  0\rightarrow & M & \rightarrow & D & \rightarrow  & G^0 & \rightarrow 0 \\
 & \downarrow &  & \downarrow &  & &  \\
 & 0 &  & 0 &  &  &  \\
\end{array}$$
From the vertical middle short exact sequence, $pd(D)\leq
pd(K)+1\leq n$. Now, consider the Gorenstein projective resolution
of $M$:
$$0\rightarrow G_n\rightarrow P_n\rightarrow ... \rightarrow P_1
\rightarrow M \rightarrow 0$$ where all $P_i$ are projective and
$G_n$ is Gorenstein projective. Devise this sequence on short exact
sequence as

$$\begin{array}{ccccccc}
                   0\rightarrow & G_1 & \rightarrow & P_1 & \rightarrow & M & \rightarrow 0\\
                  0\rightarrow & G_2 & \rightarrow & P_2 & \rightarrow & G_1 & \rightarrow 0 \\
                 & \vdots &  &  \vdots & & \vdots & \\
                   0\rightarrow & G_n & \rightarrow &P_n & \rightarrow   & G_{n-1}&\rightarrow 0 \\
                 \end{array}$$

 Clearly, by Lemma \ref{lemma1}, for all $1\leq i\leq n$,
$Gpd(G_i)\leq n-i \leq n$. \\
Consider also the following projective resolution of $G_n$:
$$...\rightarrow P_{n+2}\rightarrow P_{n+1} \rightarrow G_n
\rightarrow 0$$ and devise this long sequence on short exact
sequences as $0\rightarrow G_{i+1}\rightarrow P_{i+1} \rightarrow
G_i\rightarrow 0$ for all $i\geq n$. It is clear that for all $i\geq
n$, $G_i$ is
Gorenstein projective module (by \cite[Theorem 2.5]{Holm}).\\
 On the other hand, since $G^0$ is
Gorenstein projective, there a co-proper right projective resolution
of $G^0$
$$0\rightarrow G^0 \rightarrow P^1\rightarrow P^2 \rightarrow P^3
\rightarrow ...$$ such that for  every $i\geq 1$,
$G^i=Im(P^i\rightarrow P^{i+1})$ is Gorenstein projective. If we
devise this sequence on short exact sequence we get  $0\rightarrow
G^i\rightarrow
P^{i+1} \rightarrow G^{i+1} \rightarrow 0$ for all $i\geq 0$.\\
Briefly, we have

$$\begin{array}{ccccccc}
                    & \vdots &  &  \vdots & & \vdots &  \\
                   0\rightarrow & G^1 & \rightarrow & P^2 & \rightarrow & G^2 & \rightarrow 0\\
                  0\rightarrow & G^0 & \rightarrow & P^1 & \rightarrow & G^1 & \rightarrow 0 \\
                   0\rightarrow  &M & \rightarrow & D & \rightarrow & G^0 &  \rightarrow 0 \\
                   0\rightarrow & G_1 & \rightarrow &P_1 & \rightarrow   & M&\rightarrow 0 \\
                    0\rightarrow & G_2 & \rightarrow &P_2 & \rightarrow   & G_1&\rightarrow 0 \\
            & \vdots &  &  \vdots & & \vdots &
                 \end{array}$$

Thus, we have a sum short exact sequence $0\rightarrow N \rightarrow
Q \rightarrow N \rightarrow 0$ where $N=\oplus_{i\geq 1}G_i\oplus M
\oplus_{i\geq 0}G^i$ and $Q=\oplus_{i\geq 1}P_i\oplus D\oplus_{i\geq
1} P^i$. And clearly $pd(Q)=pd(D)\leq n$ and
$Gpd(N)=sup\{Gpd(G_i),Gpd(G^i), Gpd(M)\}\leq n$ (by
\cite[Proposition 2.19]{Holm}). Thus, by \cite[Theorem 2.20]{Holm},
$N$ is strongly $n$-Gorenstein projective module and $M$ is a
direct summand of $N$.\\

 The condition ``if" follows from
\cite[Propostion 2.19]{Holm} and Proposition \ref{n-stron implis
n-Gorenstein}.
\end{proof}

Dually, we have:

\begin{theorem}\label{direct summand SGI}
Let $M$ be an $R$-module and $n$ a positive integer. Then,
$Gid_R(M)\leq n$ if, and only if, $M$ is a direct summand of a
strongly $n$-Gorenstein injective module.
\end{theorem}
\begin{proof}
The proof is similar to the one of Theorem \ref{direct summand SGP}
by replacing the direct sum by the direct product and by using
\cite[Theorem 2.15]{Holm} and the dual of \cite[Propostion
2.19]{Holm} and \cite[Theorem 2.7]{Bennis and Mahdou1}.
\end{proof}
\begin{remark}
From the proof of Theorem \ref{direct summand SGP}, if $Gpd(M)=n$
then, there exists a strongly $n$-Gorenstein projective module $N$
such that $Gpd(N)=n$ and $M$ is a direct summand of $N$.
\end{remark}
\begin{proposition}\label{carac SGP}
For any module M and any positive integer $n$, the following are
equivalent:
\begin{enumerate}
  \item $M$ is strongly $n$-Gorenstein projective.
  \item There is an exact sequence $0\rightarrow M \rightarrow Q
  \rightarrow M \rightarrow 0$ where $pd(Q)\leq n$ and $Ext^i(M,P)=0$
  for every  module $P$ with finite projective dimension and all $i>n$.
  \item There is an exact sequence $0\rightarrow M \rightarrow Q
  \rightarrow M \rightarrow 0$ where $pd(Q)<\infty$ and $Ext^i(M,P)=0$
  for every projective module $P$ and all $i>n$.
\end{enumerate}
\end{proposition}
\begin{proof}\

$1\Rightarrow 2.$ By definition of strongly $n$-Gorenstein
projective modules, we have just to prove that for every $i>n$ and
all module $P$ with finite projective dimension we have
$Ext^i(M,P)=0$. That is clear from \cite[Theorem 2.20]{Holm} since
$Gpd(M)\leq n$ (by Proposition \ref{n-stron implis n-Gorenstein}).\\
$2\Rightarrow 3.$ Obvious.\\
$3\Rightarrow 1.$ Since $Ext^i(M,P)=0$
  for every projective module $P$ and all $i>n$, from the short exact sequence
$0\rightarrow M \rightarrow Q \rightarrow M \rightarrow 0$  we have
for all $i>n$
$$...\rightarrow 0=Ext^i(M,P)\rightarrow Ext^i(Q,P)\rightarrow
  Ext^i(Q,P)=0\rightarrow  ...$$
  Thus $Ext^i(Q,P)=0$. On the other hand, $Gpd(Q)=pd(Q)<\infty$ (by \cite[Proposition 2.27]{Holm}).
  Then, from \cite[Theorem 2.20]{Holm}, $pd(Q)=Gpd(Q)\leq n$.
  Consequently, $M$ is strongly $n$-Gorenstein projective.
\end{proof}
\begin{proposition}\label{G.gldim<}
If $G.gldim(R)<\infty$. Then:
\begin{enumerate}
  \item $M$ is strongly $n$-Gorenstein projective if, and only if,
  there exists an exact sequence $0\rightarrow M \rightarrow Q
  \rightarrow M \rightarrow 0$ where $pd(Q)\leq n$.
  \item $M$ is strongly $n$-Gorenstein injective if, and only if,
  there exists an exact sequence $0\rightarrow M \rightarrow E
  \rightarrow M \rightarrow 0$ where $id(E)\leq n$.
\end{enumerate}
\end{proposition}
\begin{proof}
$(1)$ The condition ``only if" is clear by definition of the
strongly $n$-Gorenstein projective module. So, we claim the ``if"
condition. Since $G.gldime(R)<\infty$, $Gpd(M)<\infty$. Thus, there
is an integer $k$ such that $Ext^i(M,P)=0$ for all $i>k$ and for all
projective module $P$. Thus, using the long exact sequence
$$...\rightarrow Ext^i(Q,P)\rightarrow Ext^i(M,P)\rightarrow
  Ext^{i+1}(M,P)\rightarrow Ext^{i+1}(Q,P)\rightarrow ...$$ we deduce that
  $Ext^{n+1}(M,P)=Ext^{n+j}(M,P)$ for all $j>0$ (since $pd(Q)\leq n$).
  Thus, if $j>k$. we conclude that for every projective module $P$,
  $Ext^{n+1}(M,P)=0$. Consequently, $M$ is strongly $n$-Gorenstein
  projective.\\

$(2)$ The proof is dual to $(1)$.
\end{proof}

\begin{proposition}\label{short sequence like holm}
Let  $M$ be a  strongly $n$-Gorenstein projective $R$-module ($n\geq
1$). Then, there is  an epimorphism  $\varphi:N\twoheadrightarrow M$
  where $N$ is strongly Gorenstein projective and
$K = Ker(\varphi)$ satisfies $pd(K)=Gpd(M)-1\leq  n - 1$.

\end{proposition}
\begin{proof} Assume that $M$ is  strongly $n$-Gorenstein projective
module.\\

 The proof will be  similar to the one of \cite[Theorem
2.10]{Holm}. For exactness we give the proof here. Let $0\rightarrow
N \rightarrow
  P_n\rightarrow ...\rightarrow P_1 \rightarrow M \rightarrow 0$ be
  an exact sequence where all $P_i$ are projective and  $N$ is strongly Gorenstein
  projective (the existence of this sequence is guaranties by Proposition \ref{n-stron implis n-Gorenstein}). By definition of strongly Gorenstein projective module,  there is an exact sequence
  $0\rightarrow N \rightarrow Q\rightarrow ...\rightarrow
  Q\rightarrow N \rightarrow 0$ where $Q$ is projective
  and such that the functor $Hom(-,P)$ leaves this sequence exact,
  whenever $P$ is projective. Thus, there exists homomorphisms,
  $Q\rightarrow P_{i}$ for $i=1,...,n$ and $N\rightarrow M$,
  such that the following diagram is commutative.
  $$\begin{array}{ccccccccc}
    0\rightarrow & N & \rightarrow & Q &\rightarrow ... \rightarrow& Q & \rightarrow & N & \rightarrow 0 \\
     & \parallel &  & \downarrow &    & \downarrow &  & \downarrow &  \\
    0\rightarrow & N & \rightarrow & P_n & \rightarrow ...\rightarrow& P_1 & \rightarrow & M &
    \rightarrow 0
  \end{array}$$
This diagram gives a chain map between complexes,
$$\begin{array}{ccccccc}
    0\rightarrow  & Q &\rightarrow ... \rightarrow& Q & \rightarrow & N & \rightarrow 0 \\
     & \downarrow &    & \downarrow &  & \downarrow &  \\
    0\rightarrow  & P_n & \rightarrow ...\rightarrow& P_1 & \rightarrow & M &
    \rightarrow 0
  \end{array}$$
which induces an isomorphism in homology. Its mapping cone is exact,
and all the modules in it, except for $P_1\oplus N$, which is
strongly  Gorenstein projective, are projective. Hence the kernel
$K$ of $\varphi: P_1\oplus N\twoheadrightarrow M$ satisfies
$pd(K)\leq n-1$, as desired.
 \end{proof}

\begin{proposition}\label{suite exacte}\
\begin{description}
  \item[(A)]
Let $0\rightarrow N\stackrel{\alpha}\rightarrow P
\stackrel{\beta}\rightarrow
  N'\rightarrow 0$ be an exact sequence of $R$-modules.
\begin{description}
  \item[Case1 ``P projective and $Gpd(N')=n<\infty$"]\
  \begin{enumerate}
    \item If $N'$ is strongly Gorenstein projective then so is $N$.
    \item if $n\geq 1$ and $N'$ is strongly $n$-Gorenstein
    projective then $N$ is strongly $(n-1)$-Gorenstein projective and $Gpd(N)=n-1$.
  \end{enumerate}
  \item[Case 2 ``$pd(P)=n<\infty$"]\

  If $N$ is strongly Gorenstein projective module then $N'$ is
  strongly $(n+1)$-Gorenstein projective.
\end{description}

  \item[(B)]Let $0\rightarrow N\stackrel{\mu}\rightarrow N'\stackrel{\nu}\rightarrow Q \rightarrow 0$ be an
exact sequence where $pd(Q)=n<\infty$.
\begin{enumerate}
  \item If $n>0$ and $N'$ is strongly Gorenstein projective then $N$ is strongly
  $(n-1)$-Gorenstein projective.
  \item If $Q$ is projective then $N$ is strongly Gorenstein
  projective if, and only if, $N'$ is strongly Gorenstein
  projective.
  \end{enumerate}
\end{description}

\end{proposition}
\begin{proof} $\textbf{(A)}$
\begin{description}
  \item[Case 1] \

  $(1)$ Clear.\\

  $(2)$ If $N'$ is strongly $n$-Gorenstein projective module, there
  is a short exact sequence $0\rightarrow N' \rightarrow Q \rightarrow
  N'\rightarrow 0$ where $pd(Q)\leq n$. Since $Gpd(N')=n$ we deduce
  that $pd(Q)=n$ (by Proposition \ref{n-stron implis n-Gorenstein}). On the  other hand, we have the following
  commutative diagram:
  $$\begin{array}{ccccccc}
     & 0 &  & 0 &  & 0 &  \\
     & \downarrow &  & \downarrow &  & \downarrow &  \\
    0\rightarrow & N & \rightarrow & P & \rightarrow & N' & \rightarrow 0 \\
    & \downarrow &  & \downarrow &  & \downarrow &  \\
     0\rightarrow & Q' & \rightarrow & P\oplus P & \rightarrow & Q & \rightarrow 0 \\
   & \downarrow &  & \downarrow &  & \downarrow &  \\
    0\rightarrow & N & \rightarrow & P & \rightarrow & N' & \rightarrow 0 \\
   & \downarrow &  & \downarrow &  & \downarrow &  \\
    & 0 &  & 0 &  & 0 &  \\
  \end{array}$$
  Since $P$ is projective, we get  $pd(Q')=n-1$ and since $Gpd(N')=n$
  we deduce that $Gpd(N)=n-1$ (by Lemma \ref{lemma1}). Thus, $N'$ is strongly
  $(n-1)$-Gorenstein projective (by \cite[Theorem 2.20]{Holm}).

  \item[Case 2] \

Since $N$ is strongly Gorenstein projective module there is an exact
sequence $0\rightarrow N\stackrel{u}\rightarrow Q
\stackrel{v}\rightarrow N \rightarrow 0$ where $Q$ is projective and
$Ext(N,K)=0$ for every module $K$ with finite projective dimension.
Thus, since $pd(P)<\infty$, the short sequence
$$0\rightarrow Hom(N,P)\stackrel{\circ v}\rightarrow
Hom(Q,P)\stackrel{\circ u}\rightarrow Hom(N,P)\rightarrow 0$$ is
exact. Hence, for $\alpha: N\rightarrow P$ there is a morphism
$\lambda:Q\rightarrow P$ such that $\alpha=\lambda\circ u$. Thus,
the following diagram is commutative $$\begin{array}{ccccccc}
  0\rightarrow & N & \stackrel{u}\rightarrow & Q & \stackrel{v}\rightarrow & N  & \rightarrow 0 \\
  & \alpha \downarrow &  &  \phi \downarrow&  &  \downarrow \alpha &  \\
 0\rightarrow &P& \stackrel{i}\rightarrow & P\oplus P & \stackrel{j}\rightarrow & P & \rightarrow 0 \\
\end{array}$$
where $\phi:Q\rightarrow P\oplus P$ is defined by
$\phi(q)=(\lambda(q),\alpha\circ v(q))$ and $i$ and $j$ are
respectively the canonical injection and projection. Thus, applying
the Snake Lemma, we deduce an exact sequence of the from:
$$0\rightarrow N'\rightarrow (P\oplus P)/\phi(Q)\rightarrow
N'\rightarrow 0$$ and clearly  $pd(P\oplus P/\phi(Q))\leq n+1$ and
$Gpd(N')\leq n+1$. Thus, by \cite[Theorem 2.20]{Holm}, $N'$ is
strongly (n+1)-Gorenstein projective, as desired.
\end{description}

$\textbf{(B)}$

Suppose that $N'$ is strongly Gorenstein
  projective. Thus,  there is an exact
sequence $0\rightarrow N'\stackrel{u}\rightarrow P
\stackrel{u}\rightarrow N' \rightarrow 0$ where $P$ is projective
and $Ext(N,K)=0$ for every module $K$ with finite projective
dimension. Then, similar as in $\textbf{(A) Case 2}$, there is a
morphism $\phi:P\rightarrow Q\oplus Q$ such that the following
diagram is commutative:
$$\begin{array}{ccccccc}
  0\rightarrow & N' & \stackrel{u}\rightarrow & P & \stackrel{v}\rightarrow & N'  & \rightarrow 0 \\
  & \nu \downarrow &  &  \phi \downarrow&  &  \downarrow \nu &  \\
 0\rightarrow &Q& \stackrel{i}\rightarrow & Q\oplus Q & \stackrel{j}\rightarrow & Q & \rightarrow 0 \\
\end{array}$$ Hence, applying Snake Lemma, we  get an exact  sequence of the
form  $0\rightarrow N\rightarrow Ker(\phi) \rightarrow N \rightarrow
0$.
\begin{enumerate}
  \item If $n>0$ then $pd(ker(\phi))= n-1$ and also $Gpd(N)=n-1$
  (by Lemma \ref{lemma1}). Therefore, $N$ is strongly
  $(n-1)$-Gorenstein projective (by \cite[Theorem 2.20]{Holm}).
  \item If $Q$ is projective, then $ker(\phi)$ is projective and $N$
  is Gorenstein projective. Thus, $N$ is strongly Gorenstein
  projective. Conversely, if $N$ is strongly Gorenstein projective
  it is clear that $N'\cong N\oplus P$ is strongly Gorenstein
  projective, as desired.
\end{enumerate}
\end{proof}

Dually, we have:

\begin{proposition}\label{suite exacte2}\
\begin{description}
  \item[(A)]
Let $0\rightarrow N\stackrel{\alpha}\rightarrow I
\stackrel{\beta}\rightarrow
  N'\leftrightarrow 0$ be an exact sequence of $R$-modules.
\begin{description}
  \item[Case1 ``I injective  and $Gid(N)=n<\infty$"]\
  \begin{enumerate}
    \item If $N$ is strongly Gorenstein injective then so is $N'$.
    \item If $n\geq 1$ and $N$ is strongly $n$-Gorenstein
   injective  then $N'$ is strongly $(n-1)$-Gorenstein injective  and $Gid(N')=n-1$.
  \end{enumerate}
  \item[Case 2 ``$id(P)=n<\infty$"]\

  If $N'$ is strongly Gorenstein injective module then $N$ is
  strongly $(n+1)$-Gorenstein injective.
\end{description}

  \item[(B)]Let $0\rightarrow E\stackrel{\mu}\rightarrow N'\stackrel{\nu}\rightarrow N \rightarrow 0$ be an
exact sequence where $id(E)=n<\infty$.
\begin{enumerate}
  \item If $n>0$ and $N'$ is strongly Gorenstein injective  then $N$ is strongly
  $(n-1)$-Gorenstein injective.
  \item If $E$ is injective then $N$ is strongly Gorenstein
  injective  if, and only if, $N'$ is strongly Gorenstein
  injective.
  \end{enumerate}
\end{description}

\end{proposition}

\begin{corollary}
Let $R$ be a ring. The following are equivalent:
\begin{enumerate}
  \item Every Gorenstein projective module is strongly Gorenstein
  projective.
  \item Every module such $Gpd(M)\leq 1$ is strongly 1-Gorenstein
  projective.
\end{enumerate}
\end{corollary}
\begin{proof}
Assume that every Gorenstein projective module is strongly
Gorenstein projective and consider a module $M$ such that
$Gpd(M)\leq 1$. Consider a short exact sequence $0\rightarrow N
\rightarrow P \rightarrow M \rightarrow 0$ where $P$ is projective
and so $N$ is Gorenstein projective. Hence, by the hypothesis
condition, $N$ is strongly Gorenstein projective module. Thus, by
Proposition \ref{suite exacte} (Case 2), $M$ is strongly
$1$-Gorenstein projective module, as desired.\\
Conversely, assume that every module such $Gpd(M)\leq 1$ is strongly
$1$-Gorenstein projective. Let $M$ be a Gorenstein projective
module. Thus, by the hypothesis condition $M$ is strongly
$1$-Gorenstein projective. Then, there is an exact sequence
$0\rightarrow M \rightarrow Q \rightarrow M \rightarrow 0$ where
$pd(Q)\leq 1$. Since $M$ is Gorenstein projective, so is $Q$ and
then it is projective (by \cite[Theorem 2.5 and Proposition
2.27]{Holm}). Consequently, $M$ is strongly Gorenstein projective
module.
\end{proof}
\begin{proposition}
Let $R$ be a ring. The following are equivalent:
\begin{enumerate}
  \item Every module is strongly $n$-Gorenstein projective.
  \item Every module is strongly $n$-Gorenstein injective.
\end{enumerate}
\end{proposition}
\begin{proof}
We prove only one implication and the other is similar.\\
Assume that every module is strongly $n$-Gorenstein projective. Thus
$G.gldim(R)\leq n$ (by Proposition \ref{n-stron implis n-Gorenstein}
and the hypothesis condition). Now,  consider an arbitrary module
$M$. Clearly $Gid(M)\leq n$ (since $G.gldim(R)\leq n$). Then, for
every injective module $I$, $Ext^{n+1}(I,M)=0$ (\cite[Theorem
2.22]{Holm}). On the other hand, there is an exact sequence,
$0\rightarrow M \rightarrow P \rightarrow M \rightarrow 0$ where
$pd(P)\leq n$. By \cite[Corollary 2.10]{Bennis and Mahdou2},
$id(P)\leq n$. Consequently, $M$ is strongly $n$-Gorenstein
injective, as desired.
\end{proof}
\begin{proposition}
Let $R$ be a ring with finite Gorenstein global dimension and $n$ a
positive  integer. The following are equivalent:
\begin{enumerate}
  \item $G.gldim(R)\leq n$.
  \item Every strongly Gorenstein projective module is strongly
  $n$-Gorenstein injective module.
  \item Every strongly Gorenstein injective module is strongly
  $n$-Gorenstein projective module.
\end{enumerate}
\end{proposition}
\begin{proof} We claim  that  $G.gldim(R)\leq n$ if, and only if, every strongly Gorenstein projective module is strongly
  $n$-Gorenstein injective module. The proof of the other
  equivalence is analogous. So, suppose that $G.gldim(R)\leq n$ and
  consider a strongly Gorenstein projective module $M$. For such
  module there is an exact sequence $0\rightarrow M \rightarrow P
  \rightarrow M \rightarrow 0$ where $P$ is projective. From \cite[Corollary 2.10]{Bennis and
  Mahdou2}, $id(P)\leq n$. Hence, from Proposition \ref{G.gldim<},
  $M$ is strongly $n$-Gorenstein injective.\\
  Conversely, suppose that every strongly Gorenstein projective module is strongly
  $n$-Gorenstein injective module and let $P$ be a projective
  module (then strongly Gorenstein projective). By the hypothesis
  condition, $P$ is strongly $n$-Gorenstein injective. Thus, there
  is an exact sequence $0\rightarrow P \rightarrow E \rightarrow P
  \rightarrow 0$ where $id(E)\leq n$. Hence, $P\oplus P\cong E$.
  Consequently, $id(P)\leq n$. Then, from \cite[Theorem 2.1 and Lemma 2.2]{Bennis and
  Mahdou2}, $G.gldim(R)\leq n$, as desired.
\end{proof}

\end{section}

%%%%%%%%%%%%%%%%%%%%%%%%%%%%%%%%%%%%%%%%%%%%%%%%%%%%%%%%%
%%%%%%%%%%%%%%%%%%%%%%%%%%%%%%%%%%%%%%%%%%%%%%%%%%%%%%%%%
%%Section2%%%%%%%%%%%%%%%%%%%%%%%%%%%%%%%%%%%%%%%%%%%
%%%%%%%%%%%%%%%%%%%%%%%%%%%%%%%%%%%%%%%%

\begin{section}{Strongly $n$-Gorenstein flat modules}
In this section, we introduce and study the strongly $n$-Gorenstein
flat modules which are defined as follows:
\begin{definition} An $R$-module $M$ is said to be strongly $n$-Gorenstein flat,
if there exists a short exact sequence
$$0\longrightarrow M \longrightarrow F \longrightarrow M
\longrightarrow 0$$ where $fd_R(P)\leq n$ and $Tor_R^{n+1}(M,I)=0$
whenever $I$ is injective.
\end{definition}

A direct consequence of the above definition is that, the strongly
$0$-Gorenstein flat modules are just the strongly Gorenstein flat
modules (by \cite[Proposition 3.6]{Bennis and Mahdou1}). Also every
module with finite flat dimension less or equal than $n$ is a
strongly $n$-Gorenstein flat module. Also we have:

\begin{proposition}\label{n-stron flat implis n-Gorenstein} Let $n$ be a positive integer and  $M$ be a
strongly $n$-Gorenstein flat $R$-module. Then, the following hold:
\begin{enumerate}
  \item If $0\rightarrow N \rightarrow P_n\rightarrow ...
\rightarrow P_1 \rightarrow M \rightarrow 0$ is an exact sequence
where all $P_i$ are projective, then $N$ is strongly Gorenstein flat
module and consequently $Gfd(M)\leq n$.
  \item Moreover, if $0\longrightarrow M \longrightarrow F \longrightarrow M
\longrightarrow 0$ is a short exact  sequence, where $fd(F)<
\infty$,  then $Gfd(M)=fd(F)$ and consequently $M$ is strongly
$k$-Gorenstein flat module with $k:=pd(P)$.
\end{enumerate}

\end{proposition}
\begin{proof}

\begin{enumerate}
  \item Using an $n$-step projective resolution of $M$ and \cite[Proposition
3.6]{Bennis and Mahdou1}, the proof is analogous to Proposition
\ref{n-stron implis n-Gorenstein}.
  \item  Consider an exact short sequence $(\mp)\quad 0\longrightarrow M \longrightarrow F \longrightarrow M
\longrightarrow 0$ where $fd(F)< \infty$. We claim $Gfd(M)=fd(F)$.\\
Consider an $n$ step projective resolution $$0\rightarrow N
\rightarrow P_n \rightarrow ...\rightarrow P_1 \rightarrow
M\rightarrow 0$$

From $(1)$ above $N$ is strongly Gorenstein flat module. Thus, there
is an short exact sequence $(\star)\quad 0\rightarrow N \rightarrow
P \rightarrow N  \rightarrow 0$ where $P$ is flat and $Tor(N,I)=0$
whenever $I$ is injective. Hence, from $(\star)$, for all $i>0$,
$Tor^i(N,I)=0$. So, we have $Tor^{n+i}(M,I)=Tor^i(N,I)=0$.\\
Now,  suppose that $fd(F):=k$ and let $I$ be an arbitrary injective.
From  the  short exact sequence $(\mp)$ we have the long exact
sequence
 $$....Tor^{i+1}(F,I)\rightarrow Tor^{i+1}(M,I)\rightarrow
 Tor^i(M,I)\rightarrow Tor^i(F,I)\rightarrow ...$$
 Hence, for all $i>k$,
 $Tor^i(M,I)=Tor^{i+1}(M,I)=...=Tor^{i+n}(M,I)=0$. In particular,
 $Tor^{k+1}(M,I)=0$. Consequently,  $M$ is strongly $k$-Gorenstein
 flat module. Then, from $(1)$ above $Gfd(M)\leq k=fd(F)$.\\
 Conversely, we claim
 $fd(F)\leq Gfd(M)$. Applying $Hom_{\mathbb{Z}}(-,\mathbb{Q}/\mathbb{Z})$ to the short exact
 sequence $0\rightarrow M \rightarrow F \rightarrow M \rightarrow 0$
 we get the exactness of $0\rightarrow Hom_{\mathbb{Z}}(M,\mathbb{Q}/\mathbb{Z})\rightarrow
 Hom_{\mathbb{Z}}(F,\mathbb{Q}/\mathbb{Z})\rightarrow Hom_{\mathbb{Z}}(M,\mathbb{Q}/\mathbb{Z}) \rightarrow 0$.
On the other hand, from \cite[Proposition 3.11]{Holm},
$Gid(Hom_{\mathbb{Z}}(M,\mathbb{Q}/\mathbb{Z}))\leq Gfd(M)\leq n$
and by \cite[Lemma 3.51 and Theorem 3.52]{Rotman},
$id(Hom_{\mathbb{Z}}(F,\mathbb{Q}/\mathbb{Z}))=fd(F)<\infty$. Hence,
by \cite[Theorem 2.22]{Holm} and  the injective counterpart of
Proposition \ref{carac SGP},
$Hom_{\mathbb{Z}}(M,\mathbb{Q}/\mathbb{Z})$ is strongly
$n$-Gorenstein injective module. So, from the injective counterpart
of Proposition \ref{n-stron implis n-Gorenstein} and by
\cite[Proposition 3.11]{Holm},
$$fd(F)=id(Hom_{\mathbb{Z}}(F,\mathbb{Q}/\mathbb{Z}))=Gid(Hom_{\mathbb{Z}}(M,\mathbb{Q}/\mathbb{Z}))\leq
Gfd(M)$$ Thus, we have the desired equality.
\end{enumerate}

\end{proof}

\begin{theorem}
Let $R$ be a coherent ring, $M$ be an $R$-module and $n$ be a
positive integer. Then,  $Gf(M)\leq n$ if, and only if, $M$ is a
direct summand of a strongly $n$-Gorenstein flat module.
\end{theorem}
\begin{proof} Using \cite[Theorem 3.5]{Bennis and Mahdou1},
\cite[Proposition 3.13 and Theorems 3.14 and 3.23]{Holm}, Lemma
\ref{lemma3} and Proposition \ref{n-stron flat implis n-Gorenstein}
the proof of this result is analogous to the one of  Theorem
\ref{direct summand SGP}.
\end{proof}

\begin{proposition}\label{carac SGF}
For a module M and a positive  integer $n$, the following are
equivalent:
\begin{enumerate}
  \item $M$ is strongly $n$-Gorenstein flat.
  \item There is an exact sequence $0\rightarrow M \rightarrow F
  \rightarrow M \rightarrow 0$ where $fd(F)\leq n$ and $Tor^i(M,I)=0$
  for every  module $I$ with finite injective  dimension and all $i>n$.
  \item There is an exact sequence $0\rightarrow M \rightarrow F
  \rightarrow M \rightarrow 0$ where $fd(F)<\infty$ and $Tor^i(M,I)=0$
  for every injective module $P$ and all $i>n$.
\end{enumerate}
\end{proposition}
\begin{proof}\
$1\Rightarrow 2.$ Assume that $M$ is strongly $n$-Gorenstein flat
module. Then, there is an exact sequence $0\rightarrow M \rightarrow
F \rightarrow M \rightarrow 0$ where $fd(F)\leq n$. On the other
hand, if $0\rightarrow N \rightarrow P_n \rightarrow ... \rightarrow
P_1 \rightarrow M \rightarrow 0$ is an $n$-step projective
resolution of $M$, by Proposition \ref{n-stron flat implis
n-Gorenstein}, $N$ is strongly Gorenstein flat module. Thus,
 there is a short exact sequence $(\star)\quad 0\rightarrow N
\rightarrow P \rightarrow N  \rightarrow 0$ where $P$ is flat and
$Tor(N,I)=0$ whenever $id(I)<\infty $ (from \cite[Proposition
3.6]{Bennis and Mahdou1}). Hence, from $(\star)$, for all $i>0$,
$Tor^i(N,I)=0$. So, we have $Tor^{n+i}(M,I)=Tor^i(N,I)=0$, as desired.\\

$2\Rightarrow 3.$ Obvious.\\

$3\Rightarrow 1.$ As in the proof of Proposition \ref{carac
SGP}$(3\Rightarrow 1)$, we prove that for every injective module $I$
and all $i>n$ we have $Tor^i(F,I)=0$. Suppose that $m:=fd(F)>n$ and
let $M$ be an arbitrary module. Pick a short exact sequence
$0\rightarrow M \rightarrow I \rightarrow I/M\rightarrow 0$ where
$I$ is injective. So, we have the long exact sequence
$$...\rightarrow Tor^{i+1}(F,I)\rightarrow
Tor^{i+1}(F,I/M)\rightarrow Tor^i(F,M)\rightarrow
Tor^i(F,I)\rightarrow...$$

Thus, for $i>n$ we have $Tor^i(F,M)=Tor^{i+1}(F,I/M)$. Hence,
$Tor^m(F,M)=Tor^{m+1}(F,I/M)=0$. Then,  $fd(F)\leq m-1$. Absurd.
Thus, $fd(F)\leq n$. So, it is clear that $M$ is strongly
$n$-Gorenstein flat module, as desired.
\end{proof}
\begin{proposition}\label{short sequence like holm2}
Let  $M$ be a  strongly $n$-Gorenstein flat  module over a coherent
ring $R$ ($n\geq 1$). Then, there is  an epimorphism
$\varphi:N\twoheadrightarrow M$
  where $N$ is strongly Gorenstein flat and
$K = Ker(\varphi)$ satisfies $fd(K)=Gfd(M)-1\leq  n - 1$.
\end{proposition}
\begin{proof}
Using \cite[Proposition 3.6]{Bennis and Mahdou1} and \cite[Lemma
3.17]{Holm}, the proof is analogous that in \cite[Lemma 3.17]{Holm}
and Proposition \ref{short sequence like holm}.
\end{proof}
\begin{proposition}\label{relation SGP SGF}
Let $M$ be  an $R$-module and $n$ a positive integer. Then,
following are equivalent:

\begin{enumerate}
  \item $M$ is strongly $n$-Gorenstein projective and $M$ admits a
  finite $n$-presentation.
  \item $M$ is strongly $n$-Gorenstein flat module and $M$ admits a
  finite $n+1$-presentation.
\end{enumerate}
\end{proposition}
\begin{proof}
Any way, in this Proposition $M$ admits a finite $n$-presentation.
Thus, we can consider an $n$-step free resolution $0\rightarrow N
\rightarrow F_n\rightarrow ...\rightarrow F_1\rightharpoondown
M\rightarrow 0$ where $F_i$ are finitely generated free and $N$ is
finitely generated.\\
$(1)$ If $M$ is strongly $n$-Gorenstein projective module, then $N$
is a finitely generated strongly Gorenstein projective module. Thus,
from \cite[Proposition 3.9]{Bennis and Mahdou1}, $N$ is a finitely
presented strongly Gorenstein flat module. Then, $M$ admits a finite
$(n+1)$-presentation and , for all injective module $I$, we have
$Tor^{n+1}(M,I)=Tor(N,I)=0$. On the other hand, there  is an exact
sequence $0\rightarrow M \rightarrow Q \rightarrow M \rightarrow 0$
where $fd(Q)\leq pd(Q)\leq n$. Consequently, $M$ is strongly
$n$-Gorenstein flat module which admits a finite
$n+1$-presentation.\\
$(2)$ Now, if  $M$ is strongly $n$-Gorenstein flat module which
admits a finite $(n+1)$-presentation. Then, $N$ is a finitely
presented strongly flat module. Thus, from \cite[Proposition
3.9]{Bennis and Mahdou1}, $N$ is  strongly Gorenstein projective
module. Hence, for every projective module $P$,
$Ext^{n+1}(M,P)=Ext(N,P)=0$. On the other hand, there is an exact
sequence $0\rightarrow M \rightarrow F \rightarrow M \rightarrow 0$
where $fd(F)\leq n$. But, from this short exact sequence we see that
$F$ also admits a finite $(n+1)$-presentation. Thus,
$pd(F)=fd(F)\leq n$. Consequently, $M$ is strongly $n$-Gorenstein
projective module, as desired.
\end{proof}

\begin{corollary}\label{corollary}
If $R$ is a coherent ring and $M$ a finitely presented module. Then,
$M$ is strongly $n$-Gorenstein projective if, and only if, $M$ is
strongly $n$-Gorenstein flat.
\end{corollary}

Finally, it is  clear that for a module $M$ and a positive integer $n$ we have:\\

$$`` fd(M)\leq n"\Longrightarrow ``M \;is \;strongly\; n-Gorenstein \; flat" \Longrightarrow ``Gf(M)\leq n"$$

Also, the converse are false, in general, by the same Examples
(\ref{M sgp dont implies proje} and \ref{Gp don't implis SGp}) in
section 2 since $T$ is Noetherian and $E$ is finitely presented
(since $E$ is finitely generated and $T$ is Noetherian) (by
Corollary \ref{corollary}).

\end{section}

%%%%%%%%%%%%%%%%%%%%%%%%%%%%%%%%%%%%%%%%%%%%%%%%%%%%%%%%%
%%%%%%%%%%%%%%%%%%%%%%%%%%%%%%%%%%%%%%%%%%%%%%%%%%%%%%%%%
%%%REFERENCES%%%%%%%%%%%%%%%%%%%%%%%%%%%%%%%%%%%%%%%%%%%%
%%%%%%%%%%%%%%%%%%%%%%%%%%%%%%%%%%%%%%%%%%%%%%%%%%%%%%%%

%%%%%%%%%%%%%%%%%%%%%%%%%%%%%%%%%%%%%%%%%%%%%%%%%%%%

%%%%%%%%%%%%%%%%%%%%%%%%%%%%%%%%%%%%%%%%%%%%%%%%%%%%

\bigskip\bigskip

%%%%%%%%%%%%%%%%%%%%%%%%%%%%%%%%%%%%%%%%%%%%%%%%%%%%%%%%
\end{document}